\documentclass[12pt]{amsart}
\usepackage{amsmath, amssymb, latexsym, amsthm, mathrsfs, color, mathtools, tikz,pgfplots,tikz-cd, graphicx, float, stackengine,url,accents, scrextend, enumitem}
\usepackage{calc}
\usepackage[top=2.5cm, bottom=2.5cm, left=2.5cm, right=2.5cm]{geometry}

\stackMath
\newcommand\tsup[2][2]{%
 \def\useanchorwidth{T}%
  \ifnum#1>1%
    \stackon[-1pt]{\tsup[\numexpr#1-1\relax]{#2}}{\hspace{1pt}\scriptstyle\sim}%
  \else%
    \stackon[.5pt]{#2}{\hspace{1pt}\scriptstyle\sim}%
  \fi%
}

\newcommand{\nc}{\newcommand}

\nc{\cO}{\mathcal{O}}

\nc{\ram}{\mathsf{Ramsey}}
\nc{\soo}{\mathsf{S}_1(\cO,\cO)}
\nc{\swl}{\mathsf{S}_1(\Om,\Lambda)}
\nc{\goo}{\gone(\cO,\cO)}
\nc{\gwo}{\gone(\Om,\cO)}
\nc{\gwl}{\gone(\Om,\Lambda)}
\nc{\soox}[1]{\mathsf{S}_1(\cO(#1),\cO(#1))}
\nc{\swlx}[1]{\mathsf{S}_1(\Om(#1),\Lambda(#1))}
\nc{\goox}[1]{\gone(\cO(#1),\cO(#1))}
\nc{\gwox}[1]{\gone(\Om(#1),\cO(#1))}
\nc{\gwlx}[1]{\gone(\Om(#1),\Lambda(#1))}

\nc{\sfinoo}{\mathsf{S}_{\mathrm{fin}}(\cO,\cO)}
\nc{\sfinwl}{\mathsf{S}_{\mathrm{fin}}(\Om,\Lambda)}
\nc{\sfinww}[1]{\mathsf{S}_{\mathrm{fin}}(\Om(#1),\Om(#1))}
\nc{\gfinoo}{\gfin(\cO,\cO)}
\nc{\gfinwo}{\gfin(\Om,\cO)}
\nc{\gfinwl}{\gfin(\Om,\Lambda)}
\nc{\sfinoox}[1]{\mathsf{S}_{\mathrm{fin}}(\cO(#1),\cO(#1))}
\nc{\sfinwlx}[1]{\mathsf{S}_{\mathrm{fin}}(\Om(#1),\Lambda(#1))}
\nc{\sfinwwx}[1]{\mathsf{S}_{\mathrm{fin}}(\Om(#1),\Om(#1))}
\nc{\gfinoox}[1]{\gfin(\cO(#1),\cO(#1))}
\nc{\gfinwox}[1]{\gfin(\Om(#1),\cO(#1))}
\nc{\gfinwlx}[1]{\gfin(\Om(#1),\Lambda(#1))}

\nc{\mc}{\mathcal}
\nc{\thusfar}{\my{--- Edited thus far ---}}
\nc{\lei}{\le^\oo}
\nc{\sqsubs}{\sqsubseteq^*}
\nc{\card}[1]{\left|#1\right|}
\nc{\medcard}[1]{\biggl|\,#1\,\biggr|}
\nc{\smallcard}[1]{|\,#1\,|}
\nc{\bds}{bidirectional $\roth$-scale}
\nc{\bfP}{\mathbf{P}}
\nc{\bfQ}{\mathbf{Q}}
\nc{\bbT}{\mathbb{T}}
\nc{\bbZ}{\mathbb{Z}}
\nc{\bbN}{\mathbb{N}}%{\w}
\nc{\bbC}{\mathbb{C}}%{\w}
\nc{\beq}{\begin{equation}}\nc{\eeq}{\end{equation}}
\nc{\mbq}{\mb{?}}
\nc{\mb}[1]{{\mbox{\textbf{#1}}}}
\nc{\nop}{$\times$}
\nc{\fbn}{\!\!\fbox{\!\nop\!}\!\!}
\nc{\yup}{\checkmark}
\nc{\forces}{\Vdash}
\nc{\name}[1]{\dot{#1}}
\nc{\tf}{\my{FINISHED THUS FAR}}
\nc{\FU}{Fr\'echet--Urysohn}
\nc{\gs}{$\gamma$~space}
\nc{\Gab}{\Gamma_{\mathrm{B}}}
\nc{\Omb}{\Omega_{\mathrm{B}}}
\nc{\Ga}{\Gamma}
\nc{\Om}{\Omega}
\nc{\smallbinom}[2]{\begin{psmallmatrix} #1\\ #2 \end{psmallmatrix}}
\nc{\bgamma}{\smallbinom{\Om}{\Ga}}

\nc{\productive}[2]{(#1,\allowbreak #2)^\x}
\nc{\prdct}[1]{#1^\x}
\nc{\Sel}{\mathsf{S}}
\nc{\sset}[2]{\{\,#1 : #2\,\}}
\nc{\smb}[1]{{\!\!\mb{#1}\!\!}}
\nc{\medset}[2]{{\biggl\{\,#1 : #2\,\biggr\}}}
\nc{\smallmedset}[2]{{\bigl\{\,#1 : #2\,\bigr\}}}
\nc{\set}[2]{{\left\{\,#1 : #2\,\right\}}}
\nc{\seq}[2]{{\la\, #1 : #2\,\ra}}
\nc{\eseq}[1]{#1_0, \allowbreak #1_1, \allowbreak\dotsc} %explicit sequence

\nc{\eseqint}[3]{#1_{#2}, \allowbreak\dotsc,\allowbreak #1_{#3}} %explicit sequence
\nc{\eseqstart}[2]{#1_{#2},\allowbreak #1_{#2+1},\dotsc } %explicit sequence
\nc{\eprod}[1]{#1_{1}\times \allowbreak#1_{2}\times\dotsb}

\nc{\shortprod}[1]{\prod_{n=1}^\infty{#1}_n}

\nc{\eprodint}[3]{#1_{#2}\times \allowbreak\dotsb\times\allowbreak #1_{#3}}

\nc{\seleseq}[1]{#1_1\in \mathcal{#1}_1, \allowbreak #1_2\in \mathcal{#1}_2, \allowbreak\dotsc}
\nc{\cube}{(\Cantor)^\bbN}
\nc{\Match}{\op{Match}}
\nc{\concat}[1]{\hat{\phantom{a}}\langle #1\rangle}
\nc{\poset}{\mathbb{P}}
\nc{\fn}[1]{{\op{Fn}(#1\times\w,2)}}
\nc{\linadd}{\op{linadd}}
\nc{\nonprod}{\non^\x}
\nc{\alephes}{{\aleph_0}}
\nc{\my}[1]{{\color{red}{#1}}}
\nc{\later}[1]{{\color{green} #1}}
\nc{\BTs}[1]{{\color{green} #1 (BT)}}
\nc{\Cp}{\op{C}_\mathrm{p}}
\nc{\Bp}{\op{B}_p}
\nc{\Pa}[8]{\bibitem{#1} {#2}, \emph{#3}, {#4} \textbf{#5} ({#6}), {#7}--{#8}.}
\nc{\tPa}[5]{\bibitem{#1} {#2}, \emph{#3}, {#4}, to appear.}
\nc{\sPa}[4]{\bibitem{#1} {#2}, \emph{#3}, {#4}, submitted.}
\nc{\Bc}[9]{\bibitem{#1} {#2}, \emph{#3}, in: \textbf{#4} (#5), #6 #7, #8--#9.}
\nc{\fD}{\mathfrak{D}}
\nc{\fX}{\mathfrak{X}}
\nc{\Onbd}{\Op_{\mathrm{nbd}}} %{\Op_{\mathsf{nbd}}}
\nc{\Omnb}{\Om_{\mathrm{nbd}}} %{\Om_{\mathsf{nbd}}}
\nc{\od}{\mathfrak{od}}
\nc{\Setting}[7]{\xymatrix@R=4pt@C=7pt{#1\ar@{-}[r]&#2\ar@{-}[r]&#3\\&#4\ar@{-}[u]\\
#5\ar@{-}[uu]\ar@{-}[r] & #6\ar@{-}[u]\ar@{-}[r] & #7\ar@{-}[uu]}}
\nc{\mx}[1]{\begin{matrix}#1\end{matrix}}
\nc{\plim}{p\txt{-}\lim}
\nc{\Bgp}{{\Z^\bbN}}
\nc{\Cgp}{{{\Z_2}^\bbN}}
\nc{\Cite}[1]{\textbf{[#1]}}
\nc{\Next}[1]{{#1^+}}
\nc{\cFin}{\mathrm{cF}}
\nc{\scsp}{\text{-scale space}}
\nc{\cfn}{\text{cofinal}\ }
\nc{\Con}{\text{Concentrated}}
\nc{\Lind}{\text{Lindel\"of}\,}
\nc{\con}{\text{-Concentrated}}
\nc{\lind}{\text{-Lindel\"of}\,}
\nc{\ctbl}{\text{countably }\allowbreak}
\nc{\Hur}{\text{Hurewicz}}
\nc{\intvl}[2]{{[#1(#2),\allowbreak #1(#2\!+\!1))}}
\nc{\Bdd}{\mathbf{B}}
\nc{\Dfin}{\mathfrak{D}_\mathrm{fin}}
\nc{\grbl}{{\mbox{\textit{\tiny gp}}}}
\nc{\bbP}{\mathbb{P}}
\nc{\bbM}{\mathbb{M}}
\nc{\bbQ}{\mathbb{Q}}
\nc{\bbH}{\mathbb{H}}
\nc{\BOfat}{\B_{\Om_{\mathrm{fat}}}}%\B_{\mathrm{fat}}}
\nc{\Bgood}{\B_{\mathrm{good}}}
\nc{\compactN}{\cl{\mathbb{N}}}
\nc{\blocks}[2]{\op{cl}_{#2}(#1)}
\nc{\blocksplus}[2]{\op{cl}^+_{#2}(#1)}
\nc{\arx}[1]{\texttt{http://arxiv.org/math/#1}}
\nc{\bq}{\begin{quote}}
\nc{\eq}{\end{quote}}
\nc{\cl}[1]{\overline{#1}}
\nc{\Cl}[2]{\mathrm{cl}_{#1}(#2)}
\nc{\CH}{the Continuum Hypothesis}
\nc{\MA}{Martin's Axiom}
\nc{\Bfat}{\B_\mathrm{fat}}
\nc{\inv}{^{-1}}
\nc{\Cantor}{{2^\w}}
\nc{\bP}{\mathbf{P}}
\nc{\bof}{\op{\fb}}
\nc{\dof}{\op{\fd}}
\nc{\bofF}{\bof(\cF)}
\nc{\sr}[3]{\underset{\mbox{#3}}{\mbox{#1}}}
\nc{\gp}{\binom{\Om}{\Ga}}
\nc{\gpsmall}{\mbox{$\gp$}}
\nc{\gig}{\gimel}%{\gimel\Ga}
\nc{\gns}{\sone(\Om,\gig)}
\nc{\nsr}[2]{#1}
\nc{\Srg}{{\mathbb{S}}}
\nc{\Srgs}{{\mathbb{S}^*}}
\nc{\NN}{{\w^{\w}}}
\nc{\ZN}{{\Z^{\bbN}}}
\nc{\NNup}{{\w^{\uparrow\w}}}
\nc{\NNbarup}{{\overline{\w}}^{\uparrow\w}}
\nc{\NNupb}{{b^{\uparrow\bbN}}}
\nc{\Pof}{\op{P}}
\nc{\PN}{{\Pof(\w)}}
\nc{\rothx}[1]{{[#1]^{\mbox{\tiny $\infty$}}}}
\nc{\tx}{{\tilde{x}}}
\nc{\roth}{{[\w]^{\w}}}
\nc{\roths}{{[b]^{\mbox{\tiny $\infty$}}}} 
\nc{\Fin}{\mathrm{Fin}}
\nc{\ici}{[\bbN]^{ \infty, \infty}}%{{[\bbN]^{(\aleph_0,\aleph_0)}}}
\nc{\Inc}{{\compactN^{\uparrow\bbN}}}
\nc{\powInc}[1]{{\big(\Inc\big)^{#1}}}
\nc{\powFin}[1]{{\big(\Fin\big)^{#1}}}
\nc{\powPN}[1]{{\big(\PN\big)^{#1}}}
\nc{\NcompactN}{{\compactN^\bbN}}
\nc{\Uarrow}{\smash{\big\uparrow}}
\nc{\LE}{\preccurlyeq}
\nc{\GE}{\succcurlyeq}
\nc{\op}{\operatorname}
\nc{\im}{\op{Im}}
\nc{\Span}{\op{span}}
\nc{\maxfin}{\op{maxfin}}
\nc{\ran}{\op{range}}
\nc{\iso}{\cong}
\nc{\Madd}{{\M}^\star}
\nc{\cI}{\mathcal{I}}
\nc{\cJ}{\mathcal{J}}
\nc{\scrA}{\mathscr{A}}
\nc{\scrB}{\mathscr{B}}
\nc{\scrC}{\mathscr{C}}
\nc{\scrD}{\mathscr{D}}
\nc{\scrF}{\mathscr{F}}
\nc{\scrK}{\mathscr{K}}
\nc{\A}{\D\forall}
\nc{\B}{\mathrm{B}}
\nc{\cB}{\mathcal{B}}
\nc{\cZ}{\mathcal{Z}}
\nc{\bB}{\mathbf{B}}
\nc{\BS}{\mathbf{B}(\mathcal{S})}
\nc{\BF}{\mathbf{B}(\mathcal{F})}
\nc{\BU}{\mathbf{B}(\mathcal{U})}
\nc{\cSp}{\mathcal{S}^+}
\nc{\cFp}{\mathcal{F}^+}
\nc{\cUp}{\mathcal{U}^+}

\nc{\BG}{\B_\Ga}
\nc{\BL}{\B_\Lambda}
\nc{\BT}{\B_\Tau}
\nc{\BTstar}{\B_{\Tau^*}}
\nc{\BO}{\B_\Om}
\nc{\DO}{\cD_\Om}
\nc{\KO}{\cK_\Om}
\nc{\CG}{C_\Ga}
\nc{\CL}{C_\Lambda}
\nc{\CT}{C_\Tau}
\nc{\CTstar}{C_{\Tau^*}}
\nc{\CO}{C_\Om}
\nc{\COgp}{C_{\Om^{\grbl}}}
\nc{\CLgp}{C_{\Lambda^{\grbl}}}
\nc{\BOgp}{\B_{\Om}^{\grbl}}
\nc{\BLgp}{\B_{\Lambda^{\grbl}}}
\nc{\sfC}{\mathsf{C}}
\nc{\sfD}{\mathsf{D}}
\nc{\bD}{\mathbf{D}}
\nc{\Tau}{\mathrm{T}}
\nc{\cA}{\mathcal{A}}
\nc{\cK}{\mathcal{K}}
\nc{\cD}{\mathcal{D}}
\nc{\cF}{\mathcal{F}}
\nc{\cS}{\mathcal{S}}
\nc{\cT}{\mathcal{T}}
\nc{\cG}{\mathcal{G}}
\nc{\cY}{\mathcal{Y}}
\nc{\J}{\mathcal{J}}
\nc{\cL}{\mathcal{L}}
\nc{\cM}{\mathcal{M}}
\nc{\cN}{\mathcal{N}}
\nc{\cH}{\mathcal{H}}

\nc{\Op}{\mathrm{O}}
\nc{\rmA}{\mathrm{A}}
\nc{\rmF}{\mathrm{F}}
\nc{\rmB}{\mathrm{B}}
\nc{\rmD}{\mathrm{D}}
\nc{\rmP}{\mathrm{P}}
\nc{\cC}{\mathcal{C}}
\nc{\cP}{\mathcal{P}}
\nc{\bbR}{\mathbb{R}}
\nc{\bbS}{\mathbb{S}}
\nc{\cU}{\mathcal{U}}
\nc{\cQ}{\mathcal{Q}}
\nc{\Un}{\bigcup}
\nc{\cV}{\mathcal{V}}
\nc{\cR}{\mathcal{R}}
\nc{\tcR}{\tilde{\mathcal{R}}}
\nc{\cW}{\mathcal{W}}
\nc{\Z}{{\mathbb Z}}
\nc{\Impl}{\Rightarrow}
\long\def\forget#1\forgotten{\marginpar{\textcolor{green}{Forgetting...}}}
\nc{\ft}{\mathfrak{t}}
\nc{\fb}{\mathfrak{b}}
\nc{\fc}{\mathfrak{c}}
\nc{\fd}{\mathfrak{d}}
\nc{\fg}{\mathfrak{g}}
\nc{\oo}{\infty}
\nc{\fr}{\mathfrak{r}}
\nc{\fk}{\mathfrak{k}}
\nc{\bidi}{\mathfrak{bidi}}
\nc{\fu}{\mathfrak{u}}
\nc{\fh}{\mathfrak{h}}
\nc{\fp}{\mathfrak{p}}
\nc{\fj}{\mathfrak{j}}
\nc{\fs}{\mathfrak{s}}
\nc{\w}{\omega}
\nc{\x}{\times}
\nc{\Iff}{\Leftrightarrow}

\nc{\nin}{\notin}
\nc{\cat}{\hat{\ }}
\nc{\sub}{\subseteq}
\nc{\spst}{\supseteq}
\nc{\sm}{\setminus}
\nc{\as}{\subseteq^*}%{\let\proclaim\relax}
\nc{\les}{\le^*}
\nc{\leinf}{\le^{\infty}}
\nc{\leS}{\le_S}
\nc{\leF}{\le_F}
\nc{\leU}{\le_U}
\nc{\gew}{\geq^\textrm{w}}
\nc{\rest}{\restriction}

\nc{\la}{\langle}
\nc{\ra}{\rangle}
\nc{\E}{\exists}
\nc{\dom}{\op{dom}}
\nc{\cov}{\op{cov}}
\nc{\add}{\op{add}}
\nc{\addmen}{\add(\Men{})}
\nc{\cof}{\op{cof}}
\nc{\cf}{\op{cf}}
\nc{\non}{\op{non}}
\nc{\unif}{\op{non}}
\nc{\COV}{\op{COV}}
\nc{\ADD}{\op{ADD}}
\nc{\COF}{\op{COF}}
\nc{\NON}{\op{NON}}
\nc{\supp}{\op{supp}}
\nc{\impl}{\to}
\nc{\Lp}{\mathcal{L_\p}}
\nc{\Wlog}{without loss of generality}
\newtheorem{thm}{Theorem}[section]
\nc{\bthm}{\begin{thm}} \nc{\ethm}{\end{thm}}
\newtheorem{need}[thm]{Need}
\nc{\bneed}{\begin{need}\color{dg}} \nc{\eneed}{\end{need}}
\newtheorem{prop}[thm]{Proposition}
\nc{\bprp}{\begin{prop}} \nc{\eprp}{\end{prop}}
\newtheorem{fact}[thm]{Fact}
\nc{\bfct}{\begin{fact}} \nc{\efct}{\end{fact}}
\newtheorem{prob}[thm]{Problem}
\nc{\bprb}{\begin{prob}} \nc{\eprb}{\end{prob}}
\newtheorem{lem}[thm]{Lemma}
\nc{\blem}{\begin{lem}} \nc{\elem}{\end{lem}}
\newtheorem{app}[thm]{Application}
\nc{\bapp}{\begin{app}} \nc{\eapp}{\end{app}}
\newtheorem{claim}[thm]{Claim}
\nc{\bclm}{\begin{claim}} \nc{\eclm}{\end{claim}}
\newtheorem{cor}[thm]{Corollary}
\nc{\bcor}{\begin{cor}} \nc{\ecor}{\end{cor}}
\newtheorem{conj}[thm]{Conjecture}
\nc{\bcnj}{\begin{conj}} \nc{\ecnj}{\end{conj}}
\theoremstyle{definition}
\newtheorem{defn}[thm]{Definition}
\nc{\bdfn}{\begin{defn}} \nc{\edfn}{\end{defn}}
\newtheorem{obs}[thm]{Observation}
\nc{\bobs}{\begin{obs}} \nc{\eobs}{\end{obs}}
\theoremstyle{remark}
\newtheorem{rem}[thm]{Remark}
\nc{\brem}{\begin{rem}} \nc{\erem}{\end{rem}}
\newtheorem{cnv}[thm]{Convention}
\nc{\bcnv}{\begin{cnv}} \nc{\ecnv}{\end{cnv}}
\newtheorem{exam}[thm]{Example}
\nc{\bexm}{\begin{exam}} \nc{\eexm}{\end{exam}}
\nc{\bpf}{\begin{proof}} \nc{\epf}{\end{proof}
}
\nc{\be}{\begin{enumerate}}
\nc{\ee}{\end{enumerate}}
\nc{\bi}{\begin{itemize}}
\nc{\bimy}{\my{\begin{itemize}}
\nc{\eimy}{\end{itemize}}}
\nc{\itm}{\item}
\nc{\ei}{\end{itemize}}
\nc{\Subsection}[1]{\goodbreak\subsection*{#1}}%\ \par}

\nc{\sone}{\mathsf{S}_1}
\nc{\sfin}{\mathsf{S}_\mathrm{fin}}
\nc{\ufin}{\mathsf{U}_\mathrm{fin}}
\nc{\Split}{\mathsf{Split}}

\nc{\gone}{\mathsf{G}_1}    
\nc{\tgfin}{\tilde{\mathsf{G}}_\mathrm{fin}}
\nc{\gfin}{\mathsf{G}_\mathrm{fin}}

\nc{\men}[1]{\sfin(\Op(#1),\Op(#1))}
\nc{\sch}{\ufin(\cO,\Omega)}
\nc{\rothb}{\text{Rothberger}}%\sone(\cO,\cO)}
\nc{\pmen}{\sfin(\Omega,\Omega)}
\nc{\Rothb}{\sone(\Op,\Op)}

\nc{\prothb}{\sone(\Omega,\Omega)}
\nc{\tU}{{\tilde{U}}}
\nc{\tF}{{\tilde{F}}}
\nc{\tY}{{\tilde{Y}}}
\nc{\tX}{{\tsup[1]{X}}}
\nc{\dtX}{{\tsup[2]{X}}}
\nc{\dt}[1]{{\tsup[2]{#1}}}
\nc{\td}{{\tilde{d}}}
\nc{\tz}{{\tilde{z}}}
\nc{\cfd}{\cf(\fd)}

\nc{\msep}{\sfin(\cD,\cD)}
\nc{\rsep}{\sone(\cD,\cD)}
\nc{\cft}{\sfin(\Omega_{\mathbf{0}},\Omega_{\mathbf{0}})}
\nc{\scft}{\sone(\Omega_{\mathbf{0}},\Omega_{\mathbf{0}})}

\nc{\Umen}{U\text{-Menger}}%\ufin(\cO,\Gamma_U)}
\nc{\hur}{\ufin(\cO,\Gamma)}
\nc{\tUmen}{\tU\text{-Menger}}%{\ufin(\cO,\Gamma_{\tU})}

\nc{\Men}{\text{Menger}}
\nc{\Sch}{\text{Scheepers}}

\nc{\aspst}{\prescript{*}{}{\spst}\ }
\nc{\eqs}{=^*}
\nc{\ctblOm}{\Omega_{\mathrm{ctbl}}}
\nc{\GNga}{{\smallbinom{\Om}{\Ga}}}
\nc{\ctblga}{\smallbinom{\ctblOm}{\Ga}}

\nc{\nadd}{\cN_{\mathrm{add}}}
\nc{\ball}{\mathrm{B}}
\nc{\cOunif}{\cO^{\textrm{unif}}}

\nc{\sep}{
\vspace{2cm}
\noindent
\begin{minipage}{\textwidth}
	\textcolor{red}{\rule{\textwidth}{1pt}}
\end{minipage}
}
\nc{\FS}{\op{FS}}
\nc{\sums}{\op{SS}}
\nc{\SG}{\op{SG}}
\nc{\tSG}{\op{\widetilde{SG}}}
\nc{\G}{\op{G}}
\nc{\pBM}{\op{wgM}}
\nc{\FSG}{\op{FSG}}
\nc{\M}{\op{M}}
\nc{\gM}{\op{gM}}
\nc{\FP}{\op{FP}}
\nc{\nonNadd}{\non(\nadd)}
\nc{\borga}{\Ga_\mathrm{Bor}}
\nc{\pick}{x}
\nc{\gen}{y}
\nc{\nullzind}{\sone(\{\Op_n^{\mathsf{unif}}\}_{n\in\bbN},\Ga)}
\nc{\nullzindf}[1]{\sone(\{\Op_{#1}^{\mathsf{unif}}\}_{n\in\bbN},\Ga)}

\definecolor{dg}{RGB}{42,101,24}

\nc{\myb}[1]{\textcolor{blue}{#1}}
\nc{\mydg}[1]{{\color{dg}{#1}}}
\DeclareMathOperator{\eexists}{\exists}
\DeclareMathOperator{\fforall}{\forall}
\DeclareMathOperator{\fforallstar}{\forall^*}
\nc{\Exists}[1]{\bigl(\eexists #1\bigr)}
\nc{\Forall}[1]{\bigl(\fforall #1\bigr)}
\nc{\Forallstar}[1]{\bigl(\fforallstar #1\bigr)}
\nc{\End}[1]{\bigl(#1\bigr)}
\nc{\dmo}[2]{\DeclareMathOperator{#1}{#2}}
\dmo{\Asc}{Asc}
\nc{\plusmin}{\wedge}

\nc{\cBsub}{{\cB^{\mbox{\tiny $\sub$}}}}

\nc{\Alice}{{\textsc{Alice}{}}}
\nc{\Bob}{{\textsc{Bob}}}
\nc{\BM}{\op{BM}}
\nc{\Palpha}{{\bbS_\alpha}}
\nc{\Pbeta}{{\bbS_\beta}}
\nc{\Pwtwo}{\bbS_{\w_2}}
\nc{\restrict}{\upharpoonright}
\nc{\U}{\mathcal U}
\nc{\zrost}{\bar{\w}^{\uparrow\w}}

\newcommand{\bigvid}{\hat{\ \ }}
\newcommand{\spl}{\Split}
\newcommand{\vid}{\hat{\ }}

\usepackage[all]{xy}

\title{Concentrated sets and $\gamma$-sets in the Miller model}

\subjclass[2020]{Primary: 03E35, 54D20. Secondary: 
03E75.}
\keywords{Concentrated set, Miller forcing, Hurewicz space, Rothberger space}
\thanks{The research of the first and the third authors
was funded in whole by the Austrian Science Fund
(FWF) [I 5930].
The research of the second author was funded by the National Science Center, Poland Weave-UNISONO call in the Weave programme
Project: Set-theoretic aspects of topological selections 2021/03/Y/ST1/00122
}

\author[V.~Haberl]{Valentin Haberl}
\address{Institut f\"ur Diskrete Mathematik und Geometrie, Technische Universit\"at Wien, Wiedner Hauptstrasse 8-10/104, 1040 Wien, Austria.}
\email{valentin.haberl.math@gmail.com}
\author[P.~Szewczak]{Piotr Szewczak}
\address{Institute of Mathematics, Faculty of Mathematics and Natural Science,
College of Sciences, Cardinal Stefan Wyszy\'nski University in Warsaw, W\'oycickiego 1/3,
01–938 Warsaw, Poland}
\email{p.szewczak@wp.pl}
\urladdr{http://piotrszewczak.pl}
\author[L.~Zdomskyy]{Lyubomyr Zdomskyy}
\address{Institut f\"ur Diskrete Mathematik und Geometrie, Technische Universit\"at Wien, Wiedner Hauptstrasse 8-10/104, 1040 Wien, Austria.}
\email{lzdomsky@gmail.com}
\urladdr{https://dmg.tuwien.ac.at/zdomskyy/}
\begin{document}

\begin{abstract}
Using combinatorial covering properties, we show that there is no concentrated set
of reals of size $\w_2$ in the Miller model.
The main result refutes a conjecture of Bartoszy\'nski and Halbeisen. We also prove that there are no $\gamma$-set of reals of size $\w_2$ in the Miller model.
\end{abstract} 

\maketitle

\section{Introduction}

We work in the Cantor space $\PN$.
A set $X\sub\PN$ is \emph{concentrated on a set} $D\sub \PN$ if the sets $X\sm U$ are countable for all open sets $U\sub \PN$ containing $D$.
A set $X\sub\PN$ is \emph{concentrated} if it is uncountable and concentrated on some countable set contained in $X$.
One of the remarkable features  of concentrated sets is the following
covering property introduced by  Rothberger~\cite{Rot38}
as a strengthening and a topological counterpart of strong measure zero sets introduced by Borel.
A set 
$X\sub \PN$  is \emph{Rothberger} if for any sequence $\eseq{\cU}$ of open covers of $X$, there are sets $U_0\in\cU_0, U_1\in\cU_1,\dotsc$ such that the family $\sset{U_n}{n\in\w}$ covers $X$. 

Another motivation behind concentrated sets is their connection
with so-called $K$-Lusin sets, which in their turn are 
related to classical works of Banach and Kuratowski 
on the existence of non-vanishing $\sigma$-additive measures 
defined on all subsets of the real line, see, e.g.,
\cite{BarHal03} and references therein.
Let $\roth$ and $\Fin$ be the families of all infinite and finite subsets of $\w$, respectively.
Recall that a set $X\sub\roth$ is \emph{$K$-Lusin} if it is uncountable and
the sets $X\cap K$ are countable for all compact sets $K\sub\roth$.

\bobs\label{obs:concK-Lusin}
There is a concentrated set in $\PN$ of cardinality $\kappa$ if and only if there is a $K$-Lusin set in $\roth$ of cardinality $\kappa$.
\eobs 

\bpf
($\Rightarrow$)
Let $X\sub\PN$ be a concentrated set of cardinality $\kappa$.
Eventually, adding to the set $X$ countably many points, we may assume that the set $X$ contains a countable dense set $D$ in $\PN$ such that $X$ is concentrated on $D$.
The space $\PN$ (homeomorphic with $\Cantor$) is countable dense homogeneous
(see \cite{ArhMil14} and references therein), i.e., 
for any countable dense subsets $D_0,D_1$ of $\PN$
there exists a homeomorphism $h\colon\PN\to\PN$
with $h[D_0]=D_1$. 
It follows that a homeomorphic copy $Y$ of $X$ in $\PN$ is concentrated on $\Fin$.
For any compact set $K\sub\roth$, the open set $\PN\sm K$ contains $\Fin$, and thus the set $Y\cap K$ is countable.
It follows that $Y\sm \Fin$ is a $K$-Lusin set of cardinality $\kappa$.

($\Leftarrow$) 
Let $X\sub\roth$ be a $K$-Lusin set of cardinality $\kappa$.
For any open set $U\sub \PN$ containing $\Fin$, the set $\PN\setminus U$ is compact in
$\PN$, and thus the set $X\sm U=(\PN\setminus U)\cap X$
must be countable.
It follows that the set $X\cup\Fin$ is concentrated on $\Fin$ and it has cardinality $\kappa$.
\epf
 
The main result of this paper is the non-existence of concentrated sets of cardinality $\mathfrak c=\w_2$  in the Miller model, i.e., the model obtained by iterating the poset introduced in 
\cite{Mil84} with countable supports $\w_2$ many steps over a ground model of GCH.

\bthm \label{main}
In the Miller model, there is no $K$-Lusin set in $\roth$
of size $\w_2$. Equivalently, in this model there is no concentrated set
of size $\w_2$. 
\ethm

This theorem
 refutes a conjecture indirectly
stated in the proof of \cite[Proposition 3.2]{BarHal03}.
One of the main ingredients in the proof of Theorem~\ref{main} is the fact \cite{Zdo05} that in the Miller model every Rothberger set has the following covering property of Hurewicz.
A set $X\sub \PN$ is \emph{Hurewicz} if for any sequence $\eseq{\cU}$ of open covers of $X$, there are finite families $\cF_0\sub\cU_0, \cF_1\sub\cU_2,\dotsc$ such that the family $\sset{\Un\cF_n}{n\in\w}$ is a $\gamma$-cover of $X$, i.e., it is infinite and each element $x\in X$ belongs to all but finitely many sets $\Un\cF_n$. 

In the second part of the paper we consider possible sizes of so called $\gamma$-sets in the Miller model.
A cover of a set $X\sub\PN$ is an \emph{$\w$-cover} if the set $X$ is not in the cover and every finite subset of $X$ is contained in some set from the cover.
A set $X\sub\PN$ is a \emph{$\gamma$-set} if every open $\w$-cover of $X$ contains a $\gamma$-cover.
The $\gamma$-property was introduced during investigations of local properties in functions spaces~\cite{GerNag82}.
The existence of uncountable $\gamma$-sets in $\PN$ is independent from ZFC (\cite{Laver}, \cite{galvinMiller}).
In particular, by the result of Tsaban and Orenshtein (\cite[Theorem~3.6]{ot}, \cite{sss}), in the Miller model there is a $\gamma$-set in $\PN$ of size $\w_1$.
We prove the following result.

\bthm\label{thm:gamma}
In the Miller model, there is no $\gamma$-set in $\PN$ of size $\w_2$. 
\ethm

In order to prove Theorem~\ref{thm:gamma} we use the covering property of Menger.
A set $X\sub \PN$ is \emph{Menger} if for any sequence $\eseq{\cU}$ of open covers of $X$, there are finite families $\cF_0\sub \cU_0, \cF_1\sub\cU_1,\dotsc$ such that the family $\Un_{n\in\w}\cF_n$ is a cover of $X$.
In ZFC, we have the following relations between all considered in the paper combinatorial covering properties~\cite{COC2}.
\begin{figure}[H]
\begin{tikzcd}[ampersand replacement=\&,column sep=1cm]
\makebox[\widthof{$\gamma$-property}][c]{\text{Hurewicz}}\arrow{r} \& \makebox[\widthof{Rothberger}][c]{\text{Menger}}\\
\text{$\gamma$-property} \arrow{r}\arrow{u}\& \text{Rothberger}\arrow{u}\\
\end{tikzcd}
\end{figure}
\noindent The proof of Theorem~\ref{thm:gamma} is based on a result proved in \cite{Zdo2018} stating that Menger subsets of $\PN$, in the Miller, have specific structure, described in details in Section~3.

Since we never use cardinal characteristics of the continuum 
directly in our proofs but rather some known combinatorial consequences
of (in)equalities between them, the knowledge of their definitions 
does not affect understanding of this paper. Therefore 
we refer the reader to the survey of Blass~\cite{Bla10}
instead of presenting
the definitions of cardinal characteristics  
we mention.

%%%%%%%%%%%%%%%%%%%%%%%%%%%%%%%%%
\section{Proofs}

We identify each infinite subset of $\w$ with the increasing enumeration of its elements, an element of the Baire space $\NN$.
Then $\roth$ is a subset of $\NN$ and topologies in $\roth$ inherited from $\PN$ and $\NN$ coincide.
Depending on context we refer to elements of $\roth$ as sets or functions.
For natural numbers $n,m$ with $n<m$, let $[n,m):=\sset{i}{n\leq i<m}$.
For functions $x,f\in\NN$ let $[f<g]:=\sset{n}{f(n)<g(n)}$ and we use this convention also to another binary relations on $\w$.
Let $h\in\roth$. 
A  function $f\in\roth$ is \emph{$h$-unbounded over a set $M$}, if the sets $\sset{n\in \w}{[h(n),h(n+1))\sub[x<f]}$ are infinite for all functions $x\in \NN\cap M$.
Obviously, if
 $x\leq^* f$ for all $x\in \NN\cap M$, then
 $f$ is $h$-unbounded over $M$ for any $h\in\roth$.

\bdfn
A poset $\bbP$ is \emph{mild} if for every sufficiently large $\theta$ and countable elementary submodel $M$ of $H(\theta)$ with $\bbP\in M$, if a function $f\in\roth$ is $h$-unbounded over $M$ for some $h\in\roth$, then for every $p\in P\cap M$, there is an $(M,\bbP)$-generic condition $q\in P$ with $q\leq p$ such that $q$ forces
$f$ to be $h$-unbounded over $M[\Gamma]$, where $\Gamma$ is the canonical name for a $\bbP$-generic filter.
%\my{Does this definition require something more formal as to %mention about the ground model? or it is clear?}
\edfn

Obviously,  each mild poset is proper and an iteration of finitely many mild posets is again mild.
Next, we shall prove that mild posets are  preserved by countable support iterations.
As could be expected, the proof of the aforementioned preservation
 is similar to a standard proof of the preservation of properness by countable support iterations
 (see, e.g., \cite[Lemma~2.8]{Abr10} for a very detailed presentation of the latter),
 with some additional control
on the sequence $\seq{\name{p}_i}{i\in\w}$.

We follow
 conventions regarding  notation related to iterations  made in \cite{RepZdo22}, these are pretty  standard and very similar 
 to those in \cite{Abr10}.
Let $\seq{\bbP_\alpha,\name{\bbQ}_\alpha}{\alpha<\delta}$
be an iterated forcing construction for some ordinal number $\delta$.
For ordinal numbers $\alpha_0,\alpha_1$ with $\alpha_0\leq\alpha_1\leq\delta$, we denote by $\bbP_{[\alpha_0,\alpha_1)}$ a
$\bbP_{\alpha_0}$-name for the quotient poset $\bbP_{\alpha_1}/\bbP_{\alpha_0}$, viewed naturally
as an iteration over the ordinals $\xi\in\alpha_1\setminus\alpha_0$.
For a $\bbP_{\alpha_0}$-generic filter $G$
and a $\bbP_{\alpha_1}$-name $\tau$, where $\alpha_0\leq \alpha_1\leq\delta$, we denote by
$\tau^G$ the $\bbP_{[\alpha_0,\alpha_1)}^G$-name in $V[G]$ obtained from $\tau$ by partially interpreting it
with $G$. This allows us to speak about, e.g., $\bbP_{[\alpha_1,\alpha_2)}^G$ for
$\alpha_0\leq\alpha_1\leq\alpha_2\leq\delta$  and a $\bbP_{\alpha_0}$-generic filter
$G$. For a poset $\bbP$ we shall denote the standard $\bbP$-name
for  a $\bbP$-generic filter  by $\Gamma_{\bbP}$. We shall write $\Gamma_\alpha$ instead of $\Gamma_{\bbP_\alpha}$
whenever we work with an iterated forcing construction which will be clear from the context.
Also, $\Gamma_{[\alpha_0,\alpha_1)}$ is a $\bbP_{\alpha_1}$-name whose interpretation with respect
to a $\bbP_{\alpha_0}$-generic filter $G$ is  $\Gamma_{\bbP_{[\alpha_0,\alpha_1)}^G}$, which is an element of $ V[G]$.

\blem \label{cs-pres}
If $\seq{\bbP_\alpha,\name{\bbQ}_\alpha}{\alpha<\delta}$ is a countable support iteration of
mild (hence proper) posets for some ordinal number $\delta$, then   $\bbP_\delta$ is also mild.
\elem
\begin{proof}
The proof is by induction on $\delta$.
Since the successor case is clear, assume that $\delta$ is a limit ordinal number.
Let $M$ be a countable elementary submodel of $H(\theta)$, for a sufficiently large $\theta$ with $\bbP_\delta\in M$ and $p\in\bbP_\delta\cap M$.
 Pick an increasing sequence $\seq{\delta_i}{i\in\w}$ cofinal in $\delta\cap M$,
with $\delta_i\in M$ for all $i\in\w$.

Let $\seq{\name{y}_i}{i\in\w}$ be an enumeration of all
$\bbP_\delta$-names for sequences in $\NN$ which are elements of
$M$, where every such a name appears infinitely often.
Let also
$\seq{D_i}{i\in\w}$ be an enumeration of all open dense subsets of
$\bbP_\delta$ which belong to $M$.
Fix functions $f,h\in\NN$ and assume that the function $f$ is $h$-unbounded over $M$.
  By induction on $i\in\w $ we will define a condition
$q_i\in\bbP_{\delta_i}$ and  $\bbP_{\delta_i}$-names $\name{p}_i,
\name{n}_i$  such that
\be[label=(\textit{\roman*})]
\item $\name{p}_i$ is a name for an element of $\bbP_\delta$,  $q_0\forces_{\delta_0}\name{p}_0\leq \check{p}$,
and $q_{i+1}\forces_{\delta_{i+1}}\name{p}_{i+1}\leq\name{p}_i$;
\item $q_{i+1}\restrict\delta_i=q_i$;
\item $q_i$ is $(M,\bbP_{\delta_i})$-generic;
\item $ \name{n}_i$ is a  $\bbP_{\delta_i}$-name for a  natural number bigger than $i$; and\label{itm:four}
\item $q_i \forces_{\delta_i} \End{\name{p}_i\restrict\delta_i\in\Gamma_{\delta_i}}\wedge\End{
\name{p}_i\in D_i\cap M}\wedge\End{\name{p}_i\forces_\delta [h(\name{n}_i), h(\name{n}_i+1) )\sub [\name{y}_i<f]}$.\label{itm:five}
\ee

Suppose now that we have constructed objects as above and set $q:=\bigcup_{i\in\w}q_i$.
We have $q_i=q\restrict\delta_i$ for all $i\in\w$.
Since $q_i\forces_{\delta_i} \name{p}_i\restrict\delta_i\in\Gamma_{\delta_i}$
and $q_{i+1}\forces_{\delta_{i+1}}\name{p}_{i+1}\leq\name{p}_i$ for all $i\in\w$, a standard argument yields that
$q$ is $(M,\bbP_\delta)$-generic and $q\forces_\delta\name{p}_i\in\Gamma_\delta$ for all $i\in\w$, see, e.g.,
the proof of \cite[Lemma~2.8]{Abr10} for details.
Then
\[
q\forces_{\delta}\sset{n\in\w}{[h(n), h(n+1)) \sub [\name{y}<f]}
\text{ is infinite}\]
for any $\bbP_\delta$-name
$\name{y} \in M$ for an element of $\w^\w$: 
For a $\bbP_\delta$-name $\name{a}$, let $a:=\name{a}^G$.
Fix a $\bbP_\delta$-generic filter $G$ with $q\in G$ and $i\in\w$ with $\name{y}=\name{y_i}$.
Then the condition $p_i$ lies in $G$ and by~\ref{itm:five}, and therefore we have
\[
[h({n_i}), h(n_i+1)) \sub [y<f].
\]
By~\ref{itm:four} and the fact that $\name{y}$ appears infinitely often in the sequence $\seq{\name{y_i}}{i\in\w}$, we have infinitely many $n_i\in\w$ with the above property.

Returning now to the inductive construction, fix $i\in\w$ and assume that
$q_i\in\bbP_{\delta_i}$ and  $\bbP_{\delta_i}$-names $\name{p}_i, \name{n}_i$ satisfying relevant instances of $(i)$--$(v)$ have already been constructed.
For the remaining part of the proof, let $G_{\delta_i}$ be  a $\bbP_{\delta_i}$-generic filter containing $q_i$ and for a $\bbP_{\delta_i}$-name $\name{a}$ let $a:=\name{a}^{G_{\delta_i}}$.
We have $p_i\in \bbP_\delta\cap M$.
By~\ref{itm:five}, we know that $p_i\restrict\delta_i\in G_{\delta_i}$.
In $V[G_{\delta_i}]$, let $p_i'\in M\cap D_{i+1}$ be a condition such that $p_i'\leq p_i$ and $p_i'\restrict\delta_i\in G_{\delta_i}$.
By the maximality principle, we get a $\bbP_{\delta_i}$-name $\name{p}'_i$
for a condition in $\bbP_\delta$ such that 
\[
q_i \forces_{\delta_i}\End{\name{p}'_i\leq \name{p}_i}\wedge\End{\name{p}'_i\in M\cap D_{i+1}}\wedge\End{\name{p}'_i\restrict\delta_i\in\Gamma_{\delta_i}}.
\]

Given a  $\bbP_{\delta_{i+1}}$-generic filter $R$,  construct in $V[R]$
a decreasing sequence  $\seq{r_m}{m\in\w}\in M[R]$ of  conditions in
$\bbP^R_{[\delta_{i+1},\delta)}$
below $(\name{p}'_i\restrict [\delta_{i+1},\delta))^R$
such that for some $s_m\in \w^m$ we have 
\[
r_m\forces_{\bbP^R_{[\delta_{i+1},\delta)}} \name{y}_{i+1}\restrict m=s_m.
\]
By the maximality principle, we get a sequence
$\seq{\rho_m}{m\in\w}\in M$ of $\bbP_{\delta_{i+1}}$-names for elements of
$\bbP_{[\delta_{i+1},\delta)}$ such that
\begin{equation*}\forces_{\delta_{i+1}} \End{\rho_{m+1}\leq\rho_m}\wedge \Exists{\nu_m\in \w^m}\End{\rho_m \forces_{\bbP_{[\delta_{i+1},\delta)}} 
  \nu_m\mbox =\name{y}_{i+1}\restrict m}.
\end{equation*}
In the notation used above, let
$\name{z}\in M$ be a $\bbP_{\delta_{i+1}}$-name
for $\Un_{m\in\w}\nu_m$, an element of $\w^\w$.

We have $p'_i\in \bbP_\delta\cap M\cap D_{i+1}$.
It also follows from the above that $p'_i\restrict\delta_i\in G_{\delta_i}$. For a while we shall be working in $V[G_{\delta_i}]$.
 Since
 by our inductive assumption the poset $\bbP^{G_{\delta_i}}_{[\delta_i,\delta_{i+1})}$ is mild in $V[G_{\delta_i}]$,
there exists an $(M[G_{\delta_i}], \bbP^{G_{\delta_i}}_{[\delta_i,\delta_{i+1})})$-generic condition $\pi\leq (p_i'\restrict[\delta_i,\delta_{i+1}))^{G_{\delta_i}}$
such that
$$ \pi\forces_{\bbP^{G_{\delta_i}}_{[\delta_i,\delta_{i+1})}}
 \tau:=\{n\in  \w: [h(n), h(n+1)) \sub [\name{z}^{G_{\delta_i}}<f]\}$$
 is infinite.
Let $H$ be a $\bbP^{G_{\delta_i}}_{[\delta_i,\delta_{i+1})}$-generic filter over $V[G_{\delta_i}]$ containing $\pi$,
and fix $n_{i+1}\in\tau^H\setminus (i+2)$. Thus in $V[G_{\delta_i}*H]$ we have
$$\rho_{h(n_{i+1}+1)}^{G_{\delta_i}*H} \forces_{\bbP_{[\delta_{i+1},\delta)}^{G_{\delta_i}*H}} 
[h(n_{i+1}), h(n_{i+1}+1))\sub [ \name{y}_{i+1}^{G_{\delta_i}*H}<f].$$
In $M[G_{\delta_i}]$ pick a condition $s\in M[G_{\delta_i}]\cap H$ below $(p_i'\restrict[\delta_i,\delta_{i+1}))^{G_{\delta_i}}$,
forcing the above properties of $n_{i+1}$, $\tau$, and $\rho_{h(n_{i+1}+1)}$.
By the maximality principle  we obtain $\bbP_{[\delta_i,\delta_{i+1})}^{G_{\delta_i}}$-names
$\name{s}$ and $\rho$ in $M[G_{\delta_i}]$ for some elements of $\bbP_{[\delta_i,\delta_{i+1})}^{G_{\delta_i}}$
and $\bbP_{[\delta_{i+1},\delta)}^{G_{\delta_i}}$, respectively, and a name $\name{n}_{i+1}$ for a natural number
such that

\begin{multline}\label{long}
 \pi \forces_{{\bbP_{[\delta_i,\delta_{i+1})}^{G_{\delta_i}}}} \End{\name{s}\in M[G_{\delta_i}]\cap \Gamma^{G_{\delta_i}}_{[\delta_i,\delta_{i+1})}}
\wedge
\End{\name{s}\leq
(\name{p}'\restrict [\delta_i,\delta_{i+1})^{G_{\delta_i}}}
\wedge\\
\wedge\End{\name{s}\forces_{{\bbP_{[\delta_i,\delta_{i+1})}^{G_{\delta_i}}}}
 \rho\leq (\name{p}'\restrict [\delta_{i+1},\delta))^{G_{\delta_i}}}
\wedge
\rho
\forces_{{\bbP_{[\delta_{i+1},\delta)}^{G_{\delta_i}}}}
 [h(\name{n}_{i+1}),h(\name{n}_{i+1}+1)) \sub [\name{y}_{i+1}< f].
\end{multline}

Using the maximality principle again, we can find
$\bbP_{\delta_i}$-names for the objects appearing in
equation~(\ref{long}) such that $q_i$ forces this equation. We shall
use the same notation for these names. It remains to set
$q_{i+1}:=q_i\bigvid \pi$ and
$\name{p}_{i+1}:=\name{p}'_i\restrict\delta_i\bigvid\name{s}\bigvid\rho$
and note that they together with the name $\name{n}_{i+1}$ satisfy $(i)$--$(v)$ for $i+1$.
\end{proof}

By a \emph{Miller tree} we understand a subtree $T$ of $\w^{<\w}$
consisting of increasing  finite sequences such that the following
conditions are satisfied:
\begin{itemize}
\item
 Every $t \in T$ has an extension $s\in T$ which
is splitting in $T$, i.e., there are more than one immediate
successors of $s$ in $T$;
\item If $s$ is splitting in $T$, then it
has infinitely many immediate successors in $T$.
\end{itemize}
 The \emph{Miller forcing} is the
collection $\mathbb M$ of all Miller trees ordered by inclusion,
i.e., smaller trees carry more information about the generic. This
poset was introduced in \cite{Mil84}.
 For a Miller tree $T$ we
shall denote  the set of all splitting nodes of $T$ by $\spl(T)$.
The set 
$\spl(T)$ may be written in the form $\bigcup_{i\in\w}\spl_i(T)$,
where
\[
\spl_i(T):=\{t\in\spl(T)\: :\: |\{s\in\spl(T):s\subsetneq t\}|=i\}.
\]
If $T_0, T_1\in\mathbb M$, then $T_1\leq_i T_0$ means that $T_1\leq T_0$
and $\spl_i(T_1)=\spl_i(T_0)$.
It is easy to check  that for any
sequence $\seq{T_i}{i\in\w}\in\mathbb M^\w$, if $T_{i+1}\leq_i T_i$
for all $i\in\w$, then $\bigcap_{i\in\w}T_i\in\mathbb M$.

For a node $t$ in a Miller tree $T$ we denote by $T_t$ the set
$\{s\in T : s$ is compatible with $t\}$, which is also a Miller tree.

\blem \label{mild-miller}
The Miller forcing $\mathbb M$ is mild.
\elem
\begin{proof}
Let $N$ be a countable elementary submodel of $H(\theta)$ for a sufficiently large $\theta$ with $\bbM\in N$ and $T\in\mathbb
M\cap N$.
Let $\seq{\name{y}_i}{i\in\w}$ be an enumeration of all
$\mathbb M$-names for infinite subsets of $\w$ which are elements of
$N$, in which every such a name appears infinitely often. Let also
$\seq{D_i}{i\in\w}$ be an enumeration of all open dense subsets of
$\mathbb M$ which belong to $N$.
Suppose that a function $f\in\roth$
is $h$-unbounded over $N$ for some $h\in\roth$.
We shall inductively construct a sequence $\seq{
T_i}{i\in\w}$ such that $T_{i+1}\leq_i T_i$ and
$T_\infty:=\bigcap_{i\in\w}T_i$ is as required.
Set $T_0:=T$ and
suppose that $T_i$ has already been constructed. Moreover, we shall
assume that $(T_i)_t\in N$ for all $t\in\spl_i(T_i)$.
 Let $\seq{t_j}{j\in\w}$ be a bijective enumeration of
$\spl_i(T_i)$.
For every $j,k\in\w$ with $t_j\bigvid k\in
T_i$ fix a decreasing sequence $\seq{S^{i,j,k}_{n}}{n\in\w}\in N$ of
elements of $D_i$ below $(T_i)_{t_j\vid k}$ such that each condition
$S^{i,j,k}_n$ decides some $a^{i,j,k}_n\in \w^n$ to be $\name{y}_i\restrict n$. Thus
$y^{i,j,k}:=\bigcup_{n\in\w}a^{i,j,k}_n\in \w^\w\cap N$, and hence
there is a natural number $ m^{i,j,k}\geq i$ such that
$$ [ h(m^{i,j,k}), h(m^{i,j,k}+1 )) \sub [y^{i,j,k}<f]. $$
Set
\[
T_{i+1}:=\bigcup\{ S^{i,j,k}_{h(m^{i,j,k}+1 )}:j\in\w, t_j\bigvid k\in T_i\}.
\]
This completes our inductive construction of the fusion sequence
$\seq{T_i}{i\in\w}$.

We claim that $T_\infty$ is as required. First
of all, the condition $T_\infty $ is $(N,\mathbb M)$-generic  because the
collection $\sset{ S^{i,j,k}_{h(m^{i,j,k}+1 )}}{j\in\w, t_j\bigvid k\in
T_i}$ is a subset of $D_i\cap N$ and predense below $T_{i+1}$ (and hence
also below $T_\infty$). Now fix an $\mathbb M$-name $\name{y}\in N$ for an
element of $\w^\w$ and suppose to the contrary, that there exist
$i\in\w$ and $R\leq T_\infty $  that forces
$[h(m),h(m+1)) \not\sub [\name{y}<f]$ for all $m\geq i$.
Enlarging $i$, if necessary, we may assume that
$\name{y}=\name{y}_i$.
Replacing $R$ with  a stronger condition, if
necessary, we may assume that $R\leq (T_i)_{t_j\vid k}$, where $t_j\vid k\in T_i$ for some
$i,j,k\in\w$.
But then $R\leq
S^{i,j,k}_{h(m^{i,j,k}+1 )}$, and the latter condition forces
$$ [ h(m^{i,j,k}), h(m^{i,j,k}+1 )) \sub [\name{y}_i<f],$$
which leads to a contradiction since $m^{i,j,k}$ has been
chosen to be above $i$.
\end{proof}

\bcor \label{imp_nice_conc}
Let   $G_{\w_2}$ be an
$\mathbb M_{\w_2}$-generic filter over $V$,  $x\in (\roth)^{V[G_{\w_2}]}$, and 
$\psi:(\roth)^V\to(\roth)^V$  be a function which is an element of $V$.
Then there exists an element $f\in\roth\cap V$ such that the set
\[
\sset{n\in\w}{[\psi(f)(n),\psi(f)(n+1))\sub [x< f]}
\]
is infinite.
\ecor
\begin{proof}
Suppose, contrary to our claim, that there exists a condition  $p\in G_{\w_2}$ 
and an $\bbM_{\w_2}$-name $\name{x}$ for an element  $x\in \w^\w\cap V[G_{\w_2}]$
such that
\begin{equation}\label{imp_eq}
p\forces\Forall{f\in\roth\cap V}\Forallstar{n\in\w} 
\End{\big[\psi(f)(n),\psi(f)(n+1)\big)\cap[\name{x}\geq f] \neq\emptyset}.
\end{equation}
We work in $V$ in what follows.
Let $N$ be a countable elementary submodel of $H(\w_3)$ with $p,\name{x}\in N$.
Fix $f\in \roth$
 such that $z\leq^*f$ for all $z\in \roth\cap N$.
 Then $f$ is $\psi(f)$-unbounded over $N$, and hence Lemmata~\ref{mild-miller} and \ref{cs-pres} imply that there exists 
 an $(N,\mathbb M_{\w_2})$-generic condition $q\leq p$
 which forces that $f$ is $\psi(f)$-unbounded over $N[\Gamma_{\mathbb M_{
 \w_2}}]$.
In particular, the condition $q$ forces that the set 
\[
\sset{n\in\w}{[\psi(f)(n),\psi(f)(n+1))\sub [\name{x}< f]}
\]
is infinite, which contradicts (\ref{imp_eq}).
 \end{proof}

\noindent\textbf{Remark.}
As it was established in \cite{BlaShe89},  in the Miller model there exists an ultrafilter $\mathcal F$
having a base in $V$, i.e., there exists a family  $\{F_\alpha: \alpha\in\w_1\}\in V$ such that 
\[
\mathcal F=\{X\sub\w:\Exists{\alpha\in\w_1}(F_\alpha\sub X)\}
\]
is an ultrafilter in $V[G_{\w_2}]$. 
By \cite[Theorem 3.1]{BlaMil99} there exists a function $x\in\roth$
such that $[f<x]\in\mathcal F$ for all functions $f\in\roth\cap V$.
Thus, for every function
$f\in\roth\cap V$ there exists an ordinal number $\alpha(f)<\w_1$ such that
$F_{\alpha(f)}\sub [f<x]$.
Let $\psi_{\mathcal F}(f)\in\roth\cap V$ be a function such that
$[\psi_{\mathcal F}(f)(n),\psi_{\mathcal F}(f)(n+1))\cap F_{\alpha(f)}\neq\emptyset$
for all $n\in\w$.  Thus $\psi_{\mathcal F}$ does not satisfy the conclusion 
of Corollary~\ref{imp_nice_conc}, and hence is an example showing that the condition $\psi\in V$ in the above-mentioned corollary cannot be omitted.
\hfill $\Box$
\medskip

%%%%%%%%%%%%%%%%%%%%%%%%%%%%%%%%%%%%%%%%%%%%%%%%%%%%%%%%%%%%%
%%%%%%%%%%%%%%%%%%%%%%%%%%%%%%%%%%%%%%%%%%%%%%%%%%%%%%%%%%%%%%%%
%%%%%%%%%%%%%%%%%%%%%%%%%%%%%%%%%%%%%%%%%%%%%%%%%%%%%%%%%%%%%%%%
We are in a position now to present the proof of the main result.
\medskip

\noindent\textit{Proof of Theorem~\ref{main}.}
Assume that there is a $K$-Lusin set $X\sub\roth$. 
For each subset $X'$ of $X$, the set $X'\cup \Fin\sub\PN$ is Hurewicz:
If $X'$ is countable, then the assertion is true.
Assume that $X'$ is uncountable.
Then $X'$ is a $K$-Lusin set and as in the proof of Observation~\ref{obs:concK-Lusin}, we have that $X'\cup \Fin$ is concentrated.
Each concentrated set is obviously Rothberger and in the Miller model each Rothberger set is Hurewicz. Indeed, each
Rothberger set is Hurewicz under $\mathfrak u<\mathfrak g$
by \cite[Theorem~5]{Zdo05} combined with \cite[Theorem 15]{Sch96}.
The fact that $\mathfrak u<\mathfrak g$ holds in the Miller model
is a direct consequence of
 \cite[Theorem 2]{BlaLaf89} combined with the
results of \cite{BlaShe89}.

Work in $V[G_{\w_2}]$.
Fix a function $f\in \roth$ and define $X_f:=\sset{x\in X}{x\not\les f}$.
Then the set $\sset{x\in \roth}{x\les f}$ is $\sigma$-compact and hence the set 
\[
X \sm X_f = \sset{x \in X}{x \les f} = X \cap  \sset{x\in \roth}{x\les f}
\]
is countable.
Each finite subset of $\w$ we identify with the increasing enumeration of its elements, an element of $\w^{<\w}$.  
Let $\phi_f\colon X_f\cup \Fin\to \roth$ be a function such that
\[\phi_f(y):=
\begin{cases}
[y\geq f],&\text{ if }y\in X_f,\\
\sset{n<|y|}{y(n)\geq f(n)}\cup\{|y|,|y|+1,\dotsc\},&\text{ if }y\in\Fin 
\end{cases}
\]
for all elements $y\in X_f\cup \Fin$.
Since $X_f\sub\sset{x\in X}{[x\geq f]\text{ is infinite}}$, the map $\phi_f$ is well defined.
Moreover, the map $\phi_f$ is continuous.

By the above, the set $X_f\cup \Fin$ is Hurewicz, and thus, the continuous image $\phi_f[X_f\cup \Fin]\sub [\w]^\w$ is Hurewicz as well. Any Hurewicz subspace of $[\w]^\w$
is bounded~\cite{Hur27,COC2}. 
It is well-known and easy to check that for every bounded  set
$B\sub [\w]^\w$  there exists an increasing function 
$s\in\roth$ such that
$b\cap [s(n),s(n+1))\neq\emptyset$
for every $b\in B$ and all but finitely many $n\in\w$~\cite[Lemma~2.13]{Tsa11}.  
Thus, there exists a function $\psi \colon \roth \to \roth$ such that 
for each element $y\in X_f\cup \Fin$, we have $[\psi(f)(n),\psi(f)(n+1))\cap \phi_f(y)\neq\emptyset$ for all but finitely many $n\in\w$.

By a standard argument~\cite[the proof of Lemma 5.10]{BlaShe87},
there exists  an ordinal number $\alpha<\w_2$ such that
\begin{itemize}
\item $\psi [\roth\cap V[G_\alpha]]\sub V[G_\alpha]$;
\item $\psi\restrict (\roth\cap V[G_\alpha])\in V[G_\alpha]$; 
\item 
$(X \setminus X_f)\cap V[G_\alpha] \in V[G_\alpha]$ for all $f \in V[G_\alpha]$; and
\item the map
$V[G_\alpha]\ni f\mapsto X \setminus X_f $ belongs to $V[G_\alpha]$.
\end{itemize}
Without loss of generality, we may assume that $V[G_\alpha]=V$.
Since $|X|=\w_2$ and $\card{\roth \cap V}\leq \w_1$, there is an element $x\in X\sm V$ such that $x$ is not an element of any compact set coded in $V$. 
Let $f\in\roth\cap V$ be a function from Corollary~\ref{imp_nice_conc}, applied to  $\psi \restriction (\roth \cap V)\in V$ and  $x\in X$.
Then the set 
\[
\sset{n\in\w}{[\psi(f)(n),\psi(f)(n+1))\sub [x< f]}
\]
is infinite.
Since $x$ is not an element of any compact set coded in $V$, we have
$x\in X_f$.
By the definition of $\psi$, we have
\[
[\psi(f)(n),\psi(f)(n+1))\cap [x\geq f]=[\psi(f)(n),\psi(f)(n+1))\cap \phi_f(x)\neq\emptyset
\] for all but finitely many $n\in\w$, a  contradiction
\hfill $\Box$
\medskip

\section{$\gamma$-sets}

We need the following notions and auxiliary results. 
For sets $a,b$, we write $a\as b$ if the set $a\sm b$ is finite.
A nonempty set $S\sub\roth$ is a \emph{semifilter} if for every sets $s\in S$ and $b\in\roth$ with $s\as b$, we have $b\in S$.
Let $a\in \roth$.
For a set $x\in\PN$ let 
\[
x/a:=\sset{n\in\w}{x\cap [a(n),a(n+1))\neq\emptyset}.
\]
For a set $X\sub\PN$, define $X/a:=\sset{x/a}{x\in X}$.
If $S$ is a semifilter, filter or an ultrafilter, then also $S/a$ is a semifilter, filter or an ultrafilter, respectively.

A \emph{semifilter trichotomy} is the statement that for every semifilter $S$ there is a function $a\in \roth$ such that the set $S/a$ is the Fr\'{e}chet filter of all cofinite sets in $\roth$, an ultrafilter or the full semifilter $\roth$.
The semifilter trichotomy is equivalent to the inequality $\fu< \fg$~(\cite[Theorems~7~and~10]{La92}, \cite[p.~11]{Blass90}), and thus it holds in the Miller model. 

The second alternative of this trichotomy implies that for any 
two semifilters $S_0,S_1$ for which there are $a_0,a_1\in\roth$ such that 
$S_0/a_0, S_1/a_1$ are ultrafiltes, there exists $a\in\roth $ such that
$S_0/a=S_1/a$ is an ultrafilter, see the proof of \cite[Theorem~9.22]{Bla10}.

\begin{lem} \label{gamma set lemma 1}
In the Miller model, if $X\sub\roth$ and $X\cup\Fin$ is a $\gamma$-set, then $X$ is bounded by $\w_1$-many functions in  $\roth$. 
\end{lem}

\begin{proof}
Let us consider the semifilter $S$ generated by $X$, i.e., 
\[
S= \sset{s\in\roth}{\Exists{x\in X}\End{x\as s}}.
\]
Fix a function $a\in\roth$.
By the semifilter trichotomy, the set $S/a$ is the Fr\'{e}chet filter, an ultrafilter or the full semifilter $\roth$. 

Assume that $S/a$ is the Fr\'{e}chet filter.
For each element $s\in S$, we have $s\cap [a(n),a(n+1))\neq\emptyset$ for all but finitely many $n\in\w$.
Thus, the set $S$ is bounded~\cite[Lemma~2.13]{Tsa11}.

Assume that $S/b$ is an ultrafilter for some $b\in\roth$.
In the Miller model, we have $\w_1=\fu<\fd=\w_2$.
Let $\sset{u_\alpha}{\alpha < \w_1}\sub\roth$ be a basis for an ultrafilter $U$.
By the remark made before Lemma~\ref{gamma set lemma 1}, there exists 
$a\in\roth$ such that $S/a=U/a$.
For each ordinal number $\alpha<\w_1$, let $a_\alpha:=\sset{a(n+1)}{n\in u_\alpha/a}$.
Let $s\in S$.
Since $S/a=U/a$, there is an ordinal number $\alpha<\w_1$ such that $u_\alpha/a\sub s/a$.
We have $s\les a_\alpha$.
Consequently, each function from $X$ is dominated by a function from $\sset{a_\alpha}{\alpha<\w_1}$.

Now we prove that $S/a=\roth$ doesn't hold.
Suppose contrary.
The semifilter $S$ is generated by $X$.
Then for each set $y \in \roth$ there is a set $x \in X$ such that $x/a\as y$.
We have $X/a\cup\Fin = (X\cup\Fin)/a$, and thus the set $X/a\cup \Fin$ is a continuous image of $X\cup\Fin$ and it is a $\gamma$-set.
Let $I$ be the summable ideal, i.e., the set 
\[
I:= \sset{b \in\PN}{\sum_{n \in b} \frac{1}{n} < \infty}.
\] 
Since 
\[
I=\Un_{N\in\w}\sset{b \in\PN}{\sum_{n \in b} \frac{1}{n} \leq N},
\]
the ideal $I$ is an $F_\sigma$-set in $\PN$.
Then the set $X':= (X/a\cup\Fin) \cap I$ is a $\gamma$-set~\cite[Theorem~3]{galvinMiller}.

For each natural number $n$, let $O_n:=\sset{z \in\PN}{n\notin z}$.
The family $\cO:=\sset{O_n}{n\in\w}$ is an $\w$-cover of $X'$:
Fix a finite set $F\sub X'$.
Since $\Un F \in I$, we have $\Un F\neq \w$ and there is a natural number $n$ such that $n\notin \Un F$.
Thus, $F\sub O_n$.
Since $X'$ is a $\gamma$-set, there is a set $y\in\roth$ such that the family $\sset{O_n}{n\in y}$ is a $\gamma$-cover of $X'$.
Take a set $b \in I\cap\roth$ such that $b \sub y$.
Since the semifilter $S$ is generated by $X$, there is an element $x\in X$ such that $x/a\as b$.
We have $x/a \as y$, and thus $x/a \in X'$. 
The set $x/a\cap y$ is infinite and for each natural number $n\in x/a\cap y$, we have $x/a\notin O_n$.
Consequently, the element $x/a\in X'$ doesn't belong to infinitely many sets from $\sset{O_n}{n\in y}$, a contradiction.
\end{proof}

Let $X\sub \PN$.
A set $Q\sub X$ is a \emph{$G_{\w_1}$-set} in $X$ if it is an intersection of $\w_1$-many open subsets of $X$.

\begin{lem} \label{gamma set lemma 2}
In the Miller model, each countable subset $Q$ of a $\gamma$-set $X \sub \PN$ is a $G_{\w_1}$-set in $X$.
\end{lem}

\begin{proof}
Let $D\sub\PN\sm X$ be a countable dense set in $\PN$.
Then the set $X\cup D$ is a $\gamma$-set: A union of a $\gamma$-set with a singleton is obviously a singleton, and a countable increasing union of $\gamma$-sets is again a $\gamma$-set by \cite[Corollary 14]{Jor08}. 
If the set $Q\cup D$ is a $G_{\w_1}$-set in $X\cup D$, then the set $Q$ is a $G_{\w_1}$-set in $X$.
Thus, we may assume that the set $Q$ is dense in $\PN$.
Since the space $\PN$ is countable dense homogeneous, we may also assume that $Q=\Fin$.

Let $\sset{a_\alpha}{\alpha < \w_1} \sub \roth$ be a set from Lemma~\ref{gamma set lemma 1}, applied to the set $X\sm\Fin$.
Fix an element $x\in X\sm \Fin$.
Then there is an ordinal number $\alpha<\w_1$ such that $x\les a_\alpha$.
For a natural number $m$, if $x\cap [m, a_\alpha(m)+1)=\emptyset$, then $x(m)\geq a_\alpha(m)+1>a_\alpha(m)$.
Thus, there are only finitely many $m$'s with the above property.
It follows that
\[
\Fin = \sset{x \in X}{\Forall{\alpha < \w_1}\Forall{n \in \w}\Exists{m \geq n}\End{x \cap [m, a_\alpha(m)+1) = \emptyset}}.
\]
Equivalently, we have
\[
\Fin = \bigcap_{\alpha < \w_1} \bigcap_{n \in \w} \bigcup_{m \geq n} \sset{x \in X}{x \cap [m, a_\alpha(m)+1) = \emptyset},
\]
and thus the set $\Fin$ is a $G_{\w_1}$-set in $X$.
\end{proof}

A set $X\sub\PN$ is \emph{weakly $G_{\w_1}$-concentrated} if
for every family $\mathcal{Q}$ of countable subsets of $X$ which is cofinal with respect to inclusion in the family of all countable subsets of $X$ and for every map $R\colon \cQ\to \Pof(X)$ such that  $R(Q)$ is a $G_{\w_1}$-set in $X$ containing $Q$, there is a family $\mathcal{Q}_1\sub \cQ$ of size $\w_1$ such that $X \sub \Un_{Q \in \mathcal{Q}_1} R(Q)$. The next fact was established in \cite{Zdo2018}.

\blem \label{lem:MillMen}
In the Miller model, each Menger set in $\PN$ is weakly $G_{\w_1}$-concentrated.
\elem

\begin{proof}[{Proof of Theorem~\ref{thm:gamma}}]
Let $X \sub \PN$ be a $\gamma$-set.
Let $\mathcal{Q}$ be the collection of all countable subsets of $X$ and $R\colon\cQ\to \Pof(X)$ be a map such that $R(Q):=Q$ for all sets $Q\in\cQ$.
By Lemma~\ref{gamma set lemma 2}, for each set $Q \in \cQ$ the set $R(Q)$ is a $G_{\w_1}$-set in $X$.
The set $X$ is Menger and, by Lemma~\ref{lem:MillMen}, it is  weakly $G_{\w_1}$-concentrated.
Then there is a family $\cQ_1\sub\cQ$ of size $\w_1$ such that 
\[
X \sub \Un_{Q \in \cQ_1} R(Q) = \Un_{Q \in \cQ_1} Q.
\]
We conclude that $\card{X} \leq \w_1$. 
\end{proof}

\section{Open problems}

We leave the reader with the following question asking 
whether the kind of concentrated sets whose non-existence 
in the Miller model 
we proved above, can exist at all.

\bprb
Is it consistent that there exists a set $X\sub\roth$
of size $|X|>\w_1$ such that for every $f\in\roth$
\begin{itemize}
\item
the set $\{x\in X:x\leq^* f\}$ is countable 
(i.e., $X$ is $K$-Lusin), and 
\item there exists $\psi(f)\in\roth$
such that for every $x\in X$, if $x\not\leq^* f$, then
\[
[\psi(f)(n),\psi(f)(n+1))\cap [x> f]\neq\emptyset
\]
for all but finitely many $n\in\w$?
\end{itemize}
\eprb

Another questions are related to sizes of Rothberger and Hurewicz sets in the Miller model.

\bprb
Is there, in the Miller model, a Rothberger set in $\PN$ of size $\w_2$?
\eprb

\bprb
Is there, in the Miller model, a Hurewicz set in $\PN$, without a homeomorphic copy of $\PN$ inside, of size $\w_2$?
\eprb

\section*{Acknowledgments}
We would like to thank Professor Piotr Zakrzewski who draw our attention to the paper of Bartoszy\'{n}ski and Halbeisen~\cite{BarHal03}.

\end{document}